\documentclass[12pt,a4paper]{article}
\usepackage[margin=2.6cm]{geometry}
\usepackage[T1]{fontenc}
\usepackage[utf8]{inputenc}
\usepackage[english]{babel}
\usepackage{microtype}
\usepackage[dvipsnames]{xcolor}
\usepackage{float}
\usepackage{graphicx}
\usepackage{caption}
\usepackage{subcaption}
\usepackage{soul}
\usepackage[shortlabels]{enumitem}

\usepackage{ upgreek }

\usepackage{xcolor}

\usepackage{amsmath, amsthm, amssymb, amsfonts, mathtools}
\usepackage{bbm}

\usepackage{tikz} 
\usetikzlibrary{arrows}

\usetikzlibrary{calc}
\usetikzlibrary{shapes}
\usetikzlibrary{positioning}
\usetikzlibrary{shapes.geometric}
\usetikzlibrary{arrows.meta}

\theoremstyle{plain}
\newtheorem{lemma}{Lemma}
\newtheorem{theorem}[lemma]{Theorem}
\newtheorem{proposition}[lemma]{Proposition}

\newtheorem{corollary}[lemma]{Corollary}
\newtheorem{defi}[lemma]{Definition}
\newtheorem*{proposition*}{Proposition}
\newtheorem*{T1}{Proposition~\ref{prop:onG*}}

\theoremstyle{remark}
\newtheorem{remark}[lemma]{Remark}
\newtheorem{example}[lemma]{Example}
\newtheorem{question}[lemma]{Question}

\newcommand{\1}{\mathbbm{1}}
\newcommand{\R}{\mathbb{R}}
\newcommand{\N}{\mathbb{N}}

\newcommand{\Aut}{\mathrm{Aut}}
\newcommand{\Stab}{\mathrm{Stab}}
\newcommand{\Cay}{\mathrm{Cay}}
\newcommand{\Sch}{\mathrm{Sch}}
\newcommand{\norm}[1]{\left\lVert#1\right\rVert}	
\newcommand{\acts}{\curvearrowright}

\newenvironment{acknowledgement}{\textbf{Acknowledgement.}}{}

\usepackage{hyperref}
\numberwithin{equation}{section}
\usepackage[margin=2.6cm]{geometry}

\AtEndDocument{\bigskip{
  \noindent
  \textbf{Ferenc Bencs}\\
  Alfr\'ed R\'enyi Institute of Mathematics, Budapest, Hungary\\
  \href{mailto:bencs.ferenc@renyi.mta.hu}{\texttt{bencs.ferenc@renyi.mta.hu}} \par
  \addvspace{\medskipamount}
  
  \noindent
  \textbf{Aranka Hru\v{s}kov\'{a}}\\ 
  Central European University, Budapest, Hungary\\
  Alfr\'ed R\'enyi Institute of Mathematics, Budapest, Hungary\\
  \href{mailto:aranka@renyi.mta.hu}{\texttt{aranka@renyi.mta.hu}} \par
  \addvspace{\medskipamount}
  
 \noindent
 \textbf{L\'{a}szl\'{o} M\'{a}rton T\'{o}th}\\
 École Polytechnique Fédérale de Lausanne, Switzerland\\
 \href{mailto:laszlomarton.toth@epfl.ch}{\texttt{laszlomarton.toth@epfl.ch}} 
}}
	
\begin{document}

\title{Factor-of-iid balanced orientation of non-amenable graphs}

\author{
   Ferenc Bencs
  \and
  Aranka Hru\v{s}kov\'{a}
  \and
  L\'{a}szl\'{o} M\'{a}rton T\'{o}th
}

\date{}

\maketitle

\begin{abstract}

    We show that if a non-amenable, quasi-transitive, unimodular graph $G$ has all degrees even then it has a factor-of-iid balanced orientation, meaning each vertex has equal in- and outdegree. This result involves extending earlier spectral-theoretic results on Bernoulli shifts to the Bernoulli graphings of quasi-transitive, unimodular graphs.
    
    As a consequence, we also obtain that when $G$ is regular (of either odd or even degree) and bipartite, it has a factor-of-iid perfect matching. This generalizes a result of Lyons and Nazarov beyond transitive graphs. 
\end{abstract}


\section{Introduction}

Let $G$ be a simple connected graph with all degrees even. An orientation of the edges of $G$ is \emph{balanced} if the indegree of any vertex is equal to its outdegree. When $G$ is finite, the term \emph{eulerian orientation} is often used, as such an orientation can be obtained from an eulerian cycle. Our interest lies in infinite graphs, so we prioritize the term balanced. Our main result is the following.

\begin{theorem} \label{thm:non-amenable_balanced_orientation}
Every non-amenable, quasi-transitive, unimodular graph $G$  with all degrees even has a factor-of-iid orientation that is balanced almost surely. 
\end{theorem}

The precise definitions of these notions are given in Section~\ref{section:basics}. Non-amenable means that all finite subsets of $G$ expand, quasi-transitive means $G$ has finitely many types of vertices, and unimodularity is a reversibility condition of the simple random walk on  $G$. A balanced orientation is a factor of iid if it is produced by a certain randomized local algorithm. To start with, each vertex of $G$ gets a random label from $[0,1]$ independently and uniformly. Then it makes a deterministic measurable decision about the orientation of its incident edges, based on the labelled graph that it sees from itself as a root. Neighboring vertices must make a consistent decision regarding the edge between them. To make the statements of our results less cumbersome, instead of saying ``a factor-of-iid orientation of the edges that is balanced almost surely'' we will simply say ``factor-of-iid balanced orientation''. (The naming is analogous for other decorations of vertices or edges.)
 
Obtaining combinatorial structures or certain models in statistical mechanics as factors of iid is a central topic in ergodic theory. See~\cite{lyons2017factors} and the references therein for an overview in the non-amenable setting. 

All Cayley graphs, in particular regular trees are unimodular.
For $d>1$, the $2d$-regular tree $T_{2d}$ is also non-amenable, so it is covered by Theorem~\ref{thm:non-amenable_balanced_orientation}. Note that on $T_{2d}$ there is a unique invariant random balanced orientation, which by Theorem~\ref{thm:non-amenable_balanced_orientation} is a factor of iid. Moreover, this result cannot be obtained by measurable versions of Lovász Local Lemma, see Remark~\ref{rmk:LLL}.

Our interest in balanced orientations is due to the fact that on a $2d$-regular graph a balanced orientation is a partial result towards a Schreier decoration. A \emph{Schreier decoration} of $G$ is a coloring of the edges with $d$ colors together with an orientation such that at every vertex, there is exactly one incoming and one outgoing edge of each color. It is a combinatorial coding of an action of the free group $F_d$ on the vertex set of the graphs. Every Schreier decoration gives a balanced orientation by forgetting the colors. In \cite{toth2019invariant}, the third author proved that all $2d$-regular unimodular random rooted graphs admit an invariant random Schreier decoration, and the current authors show in a parallel work \cite{ourselves} that such invariant random Schreier decorations can be obtained as a factor of iid in Euclidean grids in all dimensions greater than 1 as well as on all Archimedean (planar) lattices of even degree.

It remains an open question whether there indeed is a factor-of-iid Schreier decoration of $T_{2d}$. In \cite{thornton2020factor}, Thornton studies when graphs have factor-of-iid Cayley diagrams. Finding a Cayley diagram of a fixed group as a random decoration comes with (compared to a Schreier decoration) additional local restrictions on how the decoration should behave on loops. Nevertheless, the question of finding a Cayley diagram of $F_d$ overlaps with our interest in Schreier decorations on $T_{2d}$. Thornton has a result on factor-of-iid Cayley diagrams on non-amenable graphs (\cite[Theorem 1.7]{thornton2020factor}) that provides \emph{approximate} Cayley diagrams, but we do not allow for a small-probability local error here.

\smallskip

The proof of Theorem~\ref{thm:non-amenable_balanced_orientation} relies on two main ingredients. First, we reduce the question of finding a balanced orientation of $G$ to finding a perfect matching in an auxiliary graph $G^*$. Figure~\ref{fig:trafo} illustrates the construction, whose precise formulation is given in Section~\ref{sec:balanced_orientation}.

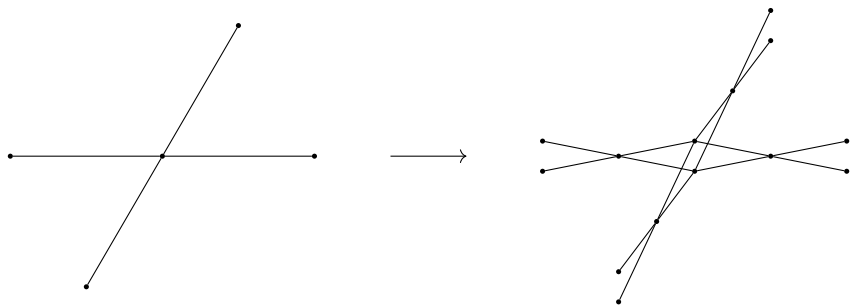
\begin{figure}[h]
    \centering
    \begin{tikzpicture}[node distance=1cm,
    main node/.style={circle,minimum width=1.9pt, fill, inner sep=0pt,outer sep = 0pt},
    red node/.style={circle,minimum width=2.2pt, fill,red, inner sep=0pt,outer sep = 0pt},
    blue node/.style={circle,minimum width=2.2pt, fill,blue, inner sep=0pt,outer sep = 0pt},
    square node/.style={draw,regular polygon,regular polygon sides=4,fill,inner sep=0.9pt,outer sep=0.1pt},
    triangle node/.style={draw,regular polygon,regular polygon sides=3,fill,inner sep=0.8pt,outer sep=0.1pt}]   
    
    \node[main node] (0) at (0,0) {};
    \node[main node] (v1) at (2,0) {};
    \node[main node] (v2) at (60:2) {};
    \node[main node] (v3) at (-2,0) {};
    \node[main node] (v4) at (240:2) {};

    \coordinate (eltolas) at (7,0) {};
    \coordinate (eltolasf)  at (0,0.2) {};
    \coordinate (eltolasa)  at (0,-0.2) {};
    
    \node[main node] (0f) at ($(0)+(eltolas)+(eltolasf)$) {};
    \node[main node] (0a) at ($(0)+(eltolas)+(eltolasa)$) {};
    
    \node[main node] (v1f) at ($(v1)+(eltolas)+(eltolasf)$) {};
    \node[main node] (v1a) at ($(v1)+(eltolas)+(eltolasa)$) {};
    
    \node[main node] (v2f) at ($(v2)+(eltolas)+(eltolasf)$) {};
    \node[main node] (v2a) at ($(v2)+(eltolas)+(eltolasa)$) {};
    
    \node[main node] (v3f) at ($(v3)+(eltolas)+(eltolasf)$) {};
    \node[main node] (v3a) at ($(v3)+(eltolas)+(eltolasa)$) {};
    
    \node[main node] (v4f) at ($(v4)+(eltolas)+(eltolasf)$) {};
    \node[main node] (v4a) at ($(v4)+(eltolas)+(eltolasa)$) {};
    
    \node[main node] (e1) at ($(1,0)+(eltolas)$) {};
    \node[main node] (e2) at ($(60:1)+(eltolas)$) {};
    \node[main node] (e3) at ($(-1,0)+(eltolas)$) {};
    \node[main node] (e4) at ($(240:1)+(eltolas)$) {};
    
    \path [draw=gray, thin] (0) edge (v1);
    \path [draw=gray, thin] (0) edge (v2);
    \path [draw=gray, thin] (0) edge (v3);
    \path [draw=gray, thin] (0) edge (v4);
    
    \path [draw=gray, thin] (0f) edge (e1);
    \path [draw=gray, thin] (0f) edge (e2);
    \path [draw=gray, thin] (0f) edge (e3);
    \path [draw=gray, thin] (0f) edge (e4);
    \path [draw=gray, thin] (0a) edge (e1);
    \path [draw=gray, thin] (0a) edge (e2);
    \path [draw=gray, thin] (0a) edge (e3);
    \path [draw=gray, thin] (0a) edge (e4);
    
    \path [draw=gray, thin] (e1) edge (v1f);
    \path [draw=gray, thin] (e1) edge (v1a);
    \path [draw=gray, thin] (e2) edge (v2f);
    \path [draw=gray, thin] (e2) edge (v2a);
    \path [draw=gray, thin] (e3) edge (v3f);
    \path [draw=gray, thin] (e3) edge (v3a);
    \path [draw=gray, thin] (e4) edge (v4f);
    \path [draw=gray, thin] (e4) edge (v4a);
    
    \draw[->] (3,0) -- (4,0);
    
    \end{tikzpicture}
    
    \caption{Obtaining $G^*$ from $G$ -- the combinatorial transformation around a vertex of degree 4}
    \label{fig:trafo}
\end{figure}

Second, we apply earlier matching results of Lyons and Nazarov \cite{LyonsNazarov}, who proved that bipartite, non-amenable Cayley graphs have a factor-of-iid perfect matching. Csóka and Lippner extended this to all non-amenable Cayley graphs in \cite{csoka2017invariant}. 

In order to deduce Theorem~\ref{thm:non-amenable_balanced_orientation} from these matching results, we have to establish vertex expansion in the appropriate Bernoulli graphing (see subsection~\ref{subsec:bernoulli_graphing} for the definition). We do this via spectral theory, and state our spectral-theoretic result here because we believe it is of interest in itself.

\begin{theorem} \label{thm:spectral_gap}
Let $G$ be a connected, unimodular, quasi-transitive graph. If $G$ is non-amenable then its Bernoulli graphing $\mathcal{G}$ has positive spectral gap. 
\end{theorem}

The interpretation of spectral gap is slightly different depending on the Bernoulli graphing being measurably bipartite or not. See Theorems~\ref{thm:non_bipart_spectral_gap} and~\ref{thm:bipart_spectral_gap} for exact statements.

Our proof of Theorem~\ref{thm:spectral_gap} requires more sophistication than simply repeating earlier arguments in a more general setting. To emphasize this, we point out that (unlike in the transitive case) $-1$ can indeed be part of the spectrum. Also our proof does not bound $\norm{\mathcal{M}}$ above by $\norm{M_G}$, where $\norm{M_G}$ is the operator norm of the Markov operator $M_G$ on $\ell^2$(V(G)), while for Cayley graphs, one has $\norm{\mathcal{M}}\leq\norm{M_G}$. 

As a consequence of Theorem~\ref{thm:spectral_gap}, we also obtain the following generalisation of the result of Lyons and Nazarov.

\begin{corollary} \label{cor:quasi_transitive_matching}
Let $G$ be a connected, unimodular, quasi-transitive non-amenable regular bipartite graph. Then $G$ has a factor-of-iid perfect matching.
\end{corollary}

The bipartite assumption in Corollary~\ref{cor:quasi_transitive_matching} cannot be dropped because there are unimodular, quasi-transitive regular graphs that have no perfect matching at all, see Remark~\ref{rmk:quasi_trans_no_perfect_matching}. 
Regularity cannot be dropped either, as for example bi-regular trees (of two different degrees of regularity) have no factor-of-iid perfect matching. 

\smallskip

We also discuss how far Theorem~\ref{thm:non-amenable_balanced_orientation} goes towards obtaining Schreier decorations on the regular tree $T_{2d}$. 

\begin{proposition} \label{prop:observations}
Regarding Schreier decorations of $T_{2d}$ we observe the following.
\begin{enumerate}[(i)]
    \item \label{itm:if_edge_coloring_then_sch_dec}
    If $T_d$ has a factor-of-iid proper edge coloring with $d$ colors then $T_{2d}$ has a factor-of-iid Schreier decoration.
    
    \item \label{itm:colourblind_Schreier}
    $T_{2d}$ has a factor-of-iid Schreier decoration with the last two colors unordered.
    
    \item \label{itm:orientation_not_finishable}
    Let $\Vec{T_4}$ denote the tree $T_4$ with edges oriented in a balanced way. ($\Vec{T_4}$ is unique up to isomorphism.) There is no $\Aut(\Vec{T_4})$-factor-of-iid Schreier decoration of $\Vec{T_4}$ with the additional property that after forgetting the colors, it coincides with the original orientation of $\Vec{T_4}$.
    
    \item \label{itm:if_sch_dec_then_bigger_sch_dec}
    For every positive integer $d$, if $T_{2d}$ has a factor-of-iid Schreier decoration then so does $T_{2d+2}$.
\end{enumerate}
\end{proposition}

It is an open question whether $T_d$ (for $d>2$) has a factor-of-iid proper edge coloring by $d$ colors. Part~\ref{itm:colourblind_Schreier} utilizes the partial result towards such a factor-of-iid proper edge coloring presented in \cite{lyons2017factors}; see subsection~\ref{subsec:proper_colouring_vs_Schreierization} for further comments. Note, however, that by part~\ref{itm:orientation_not_finishable}, obtaining a factor-of-iid Schreier decoration of $T_{4}$ cannot be achieved by selecting a balanced orientation first and then choosing the colors without modifying the orientation. This observation is unique to degree 4 because it relies on the $2$-regular tree, otherwise known as the bi-infinite path, not having a factor-of-iid proper edge coloring with two colors. For higher degree, a construction might be finished this way, as pointed out in part~\ref{itm:if_edge_coloring_then_sch_dec}.

Finally, regarding the auxiliary graph $G^*$ we show that existences of different factors of iid are equivalent.

\begin{proposition}\label{prop:onG*}
For every $2d$-regular graph $G$, the bipartite graph $G^*$ is also $2d$-regular, and the following are equivalent.
\begin{enumerate}
\item $G^*$ has got a factor-of-iid proper edge $2d$-coloring.
\item $G^*$ has got a factor-of-iid perfect matching.
\item $G^*$ has got a factor-of-iid Schreier decoration.
\end{enumerate}
Moreover, if any of these is a finitary factor, the others are too. 
\end{proposition}

For the definition of \emph{finitary} factors see subsection~\ref{subsec:fiid}.

The structure of the paper is as follows. In Section~\ref{section:basics}, we introduce the necessary notions and existing results. In Section~\ref{section:spectral_gap}, we prove our spectral-theoretic result, Theorem~\ref{thm:spectral_gap}. We deduce Corollary~\ref{cor:quasi_transitive_matching} in Section~\ref{section:perfect_matchings}, and in Section~\ref{sec:balanced_orientation}, we prove our main result, Theorem~\ref{thm:non-amenable_balanced_orientation}. Our results on other types of decorations are collected in Section~\ref{section:decorations}. Section~\ref{section:open_questions} lists some open questions.

\smallskip

{\bf Addendum.} After our manuscript was made available online, Riley Thornton brought to our attention that his Theorem~2.8 in \cite{thornton2020orienting} provides a balanced orientation in $2d$-regular graphings with expansion. Since Backhausz, Szegedy, and Virág show in \cite[Theorem 2.2]{backhausz2015ramanujan} that the Bernoulli graphing of $T_{2d}$ does have expansion, a factor-of-iid balanced orientation of $T_{2d}$ can also be obtained by combining these two results.

\begin{acknowledgement}
The authors would like to thank Miklós Abért, Jan Grebík, Matthieu Joseph, Gábor Kun, Gábor Pete, and Václav Rozhoň for inspiring discussions about various parts of this work.
The first author is supported by the NKFIH (National Research, Development and Innovation Office, Hungary) grant KKP-133921.
\end{acknowledgement}

\section{Notation and basics}\label{section:basics}

Some of the descriptions in this section are identical to the ones in our parallel work \cite{ourselves}.

\subsection{Graphs}

A graph $G$ is given by its vertex set $V(G)$ and edge set $E(G)$, where $E(G)\subset V(G)^{(2)}$ is a collection of 2-element subsets of $V(G)$ and we write $uv$ for the subset $\{u,v\}$. For any subset $A\subset V(G)$, we denote by $N_G(A)$ the neighborhood of $A$, that is $\{u\in V(G) : \exists v\in A \text{ such that } uv\in E(G)\}$. We use the calligraphic $\mathcal{G}$ for graphs that have a probability measure associated to them that makes them a graphing (see subsection~\ref{subsec:gphings} for precise definition).

\subsection{Amenability}

Let $G$ be a locally finite connected graph, and let $p_n(x,y)$ denote the probability of the simple random walk started from $x$ reaching $y$ in $n$ steps.
Then the value $\limsup_{n \to \infty} \sqrt[n]{p_n(x,y)}$ is independent of the choice of $x$ and $y$,
and is in fact equal to the norm of the Markov operator $M: \ell^2(V(G), m_{\textrm{st}}) \to \ell^2(V(G), m_{\textrm{st}})$. Here $m_{\textrm{st}}$ is the degree-biased version of the counting measure, i.e.\ $m_{\textrm{st}}(X) = \sum_{v \in X} \deg(v)$, which is a stationary measure with respect to the random walk.
The operator $M$ is defined by
\[\left(M(f)\right)(v)= \frac{1}{\deg(v)}\sum_{uv \in E(G)} f(u).\]
$M$ is self-adjoint and has norm at most 1 for any $G$. We will denote its norm (and spectral radius) by $\rho$:
\[\rho=||M||=\limsup_{n \to \infty} \sqrt[n]{p_n(x,y)}.\]
We say $G$ is \emph{amenable} if $\rho=1$ and \emph{non-amenable} if $\rho <1$.

\subsection{Schreier graphs}

Given a finitely generated group $\Gamma= \langle S \rangle$ and an action $\Gamma \acts X$ on some set $X$, the \emph{Schreier graph} $\Sch(\Gamma \acts X,S)$ of the action is defined as follows. The set of vertices is $X$, and for every $x\in X$, $s \in S$, we introduce an oriented $s$-labelled edge from $x$ to $s.x$.

Rooted connected Schreier graphs of $\Gamma$ are in one-to-one correspondence with pointed transitive actions of $\Gamma$, which in turn are in one-to-one correspondence with subgroups of $\Gamma$.
Trivially, a graph with a Schreier decoration is a Schreier graph of the free group $F_d$ on $d$ generators.
A special case is the
(left) \emph{Cayley graph} of $\Gamma$, denoted $\Cay(\Gamma,S)$, which is the Schreier graph of the (left) translation action $\Gamma \acts \Gamma$.

\subsection{Factors of iid} \label{subsec:fiid}
Let $\Gamma$ be a group. A \emph{$\Gamma$-space} is a measurable space $X$ with an action $\Gamma \acts X$. A map $\Phi: X \to Y$ between two $\Gamma$-spaces is a \emph{$\Gamma$-factor} if it is measurable and $\Gamma$-equivariant, that is $\gamma.\Phi(x)=\Phi(\gamma.x)$ for every $\gamma \in \Gamma$ and $x \in X$.

A measure $\mu$ on a $\Gamma$-space $X$ is \emph{invariant} if $\mu(\gamma.A)=\mu(A)$ for all $\gamma\in\Gamma$ and all measurable $A \subseteq X$. We say an action $\Gamma \acts (X,\mu)$ is \emph{probability-measure-preserving} (p.m.p.) if $\mu$ is a $\Gamma$-invariant probability measure. 

Let $G$ be a graph and $\Gamma \leq \Aut(G)$. Let ${\tt u}$ denote the Lebesgue measure on $[0,1]$. We endow the space $[0,1]^{V(G)}$ with the product measure ${\tt u}^{V(G)}$. The translation action $ \Gamma \acts [0,1]^{V(G)}$ is defined by \[(\gamma.f)(v) = f(\gamma^{-1}.v), \ \forall \gamma \in \Gamma, v\in V(G).\]
The action $\Gamma \acts ([0,1]^{V(G)}, {\tt u}^{V(G)})$ is p.m.p.

An orientation of $G$ can be thought of as a function on $E(G)$ sending every edge to one of its endpoints. Viewed like this, orientations of $G$ form a measurable function space ${\tt Or}(G)$ on which $\Gamma$ acts. The set ${\tt BalOr}(G) \subseteq {\tt Or}(G)$ of balanced orientations is $\Gamma$-invariant and measurable, so it is a $\Gamma$-space in itself. Similarly, the set of all Schreier decorations of $G$ forms the $\Gamma$-space ${\tt Sch}(G)$. 

\begin{defi}
A $\Gamma$-factor of iid balanced orientation (respectively, Schreier decoration) of a graph $G$ is a $\Gamma$-factor $\Phi: ([0,1]^{V(G)}, {\tt u}^{V(G)}) \to {\tt BalOr}(G)$ (respectively, to ${\tt Sch}(G)$). If the subgroup $\Gamma \leq \Aut(G)$ is not specified, we mean an $\Aut(G)$-factor.
\end{defi}

\begin{remark}
We allow $\Phi$ to not be defined on a ${\tt u}^{V(G)}$-null subset $X_0 \subseteq [0,1]^{V(G)}$. 
\end{remark}

Let us now recall some special classes of iid processes on graphs. For a fixed vertex $x \in V(G)$, let $\big(\Phi(\omega)\big)(x)$ denote the restriction of $\Phi(\omega)$ to the edges incident to $x$. We say $\Phi$ is a \emph{finitary} factor of iid if for almost all $\omega\in [0,1]^{V(G)}$, there exists an $R \in \mathbb{N}$ such that $\big(\Phi(\omega)\big)(x)$ is already determined by $\omega|_{B_G(x, R)}$. That is, if we change $\omega$ outside $B_G(x, R)$, the decoration $\Phi(\omega)$ does not change around $x$. This radius $R$ can depend on the particular $\omega$. If it does not then we say $\Phi$ is a \textit{block factor}.

When constructing factors of iid algorithmically, one often makes use of the fact that a uniform $[0,1]$ random variable can be decomposed into countably many independent uniform $[0, 1]$ random variables. In practice, this means that we can assume that a vertex has multiple labels or that a new independent random label is always available after a previous one was used.

We will use a reverse operation as well: the
composition of countably many uniform $[0,1]$ random variables is again a uniform $[0,1]$ random variable. 

\begin{remark} \label{rmk:LLL}
Note that balanced orientations of $T_{2d}$ have the property that fixing the orientation on all edges at distance $r$ from some vertex $u$ determines the orientation of edges incident to $u$, independently of $r$. Consequently, the balanced orientation constructed in Theorem~\ref{thm:non-amenable_balanced_orientation} has no local reduction to the Lovász Local Lemma (LLL). Indeed, by \cite[Section 11.1]{balliu2020classification} it has randomized local complexity $\Theta(\log n)$, whereas the algorithm of \cite{fischer2017sublogarithmic} implies $o(\log n)$ complexity for problems that have local reductions to the LLL. So although there are measurable versions of the LLL \cite{kun2013expanders, bernshteyn2019measurable}, factor-of-iid balanced orientations of $T_{2d}$ cannot be obtained that way.
\end{remark}

\subsection{Unimodular quasi-transitive graphs} \label{subsec:unimod_qt_graphs}
Unimodular random rooted graphs are central objects in sparse graph limit theory because they can represent limits of locally convergent sequences of finite graphs. In this paper, however, we only deal with a special case, namely unimodular quasi-transitive graphs.
For a thorough treatment of the topic and the connection to unimodular random rooted graphs, we refer the reader to \cite[Chapter 8.2]{lyons2017probability} and \cite{AldousLyons}.

Let $G$ be a locally finite graph, $\Gamma = \Aut(G)$. There is a function $\mu:V(G) \to \R^+$ such that for any $x,y \in V(G)$, we have \[\frac{\mu(x)}{\mu(y)} = \frac{|\Stab_{\Gamma}(x).y|}{|\Stab_{\Gamma}(y).x|}.\]

The function $\mu$ is unique up to multiplication by a constant. We say $G$ is \emph{unimodular} if $|\Stab_{\Gamma}(x).y| = |\Stab_{\Gamma}(y).x|$ for any pair $x,y \in V(G)$ that are in the same $\Gamma$ orbit, that is $y \in \Gamma.x$. So $G$ is unimodular if and only if $\mu(y)=\mu(x)$ for any $y \in \Gamma.x$.

Moreover, when $\{o_i\}$ is the orbit section of $G$ and $\sum_i\mu(o_i)^{-1}<\infty$, then we can normalize $\mu$ to obtain a probability measure on $\{o_i\}$.

In particular, when $G$ is quasi-transitive, let $T=\{o_1,\dots,o_t\}\subset V(G)$ be a set of representatives of the orbits of $\Gamma \acts V(G)$. Let $p$ be the normalized version of $\mu^{-1}$ as above -- we think of $p$ as a distribution of a random root in $G$. 

The notion of unimodularity comes hand in hand with the Mass Transport Principle. In our case, it takes the following form:

\begin{proposition}[Mass Transport Principle, Corollary 8.11. in \cite{lyons2017probability}] \label{prop:mass_transport}
Given a function $f: V(G) \times V(G) \to [0,\infty]$ that is invariant under the diagonal action of $\Gamma$, we have 

\[\sum_{i=1}^tp(o_i) \sum_{z \in V(G)} f(o_i,z) = \sum_{i=1}^tp(o_i) \sum_{z \in V(G)} f(z,o_i).\]
\end{proposition}

We immediately use the Mass Transport Principle to set up a finite state Markov chain mimicking the transitions of the random walk on $G$ between $\Gamma$-orbits.

\begin{lemma} \label{lemma:unimod_is_reversible}
For any $1\le i\neq j \le t$, the function $p$ satisfies, 
\[
    p(o_i) \left|\{o_iv\in E~|~ v \in \Gamma.o_j\}\right| = p(o_j)\left|\{vo_j\in E~|~ v \in \Gamma.o_i)\}\right|.
\]
\end{lemma}

\begin{proof}
For fixed $i\neq j$, set up a payment function $f$ with $f(x,y) = 1$ if $xy \in E(G)$, $x \in \Gamma.o_i$ and $y \in \Gamma.o_j$. Set $f(x,y)=0$ otherwise. The Mass Transport Principle gives the desired equality. 
\end{proof}

We define a Markov chain $M_T$ with states $T$ and transition probabilities \[p_{M_T}(o_i, o_j)=\frac{ \left|\{o_iv\in E(G)~|~ v \in \Gamma.o_j\}\right| }{\deg(o_i)}.\]

With slight abuse of notation, we will also denote the transition matrix by $M_T$. We write $\widetilde{p}$ for the degree-biased version of the root distribution $p$, that is

\[\widetilde{p} (o_i) = \frac{\deg(o_i)}{\Delta} \cdot p(o_i).\]

Here $\Delta = \mathbb{E}_{p}[\deg(o_i)]$ is the expected degree of a root picked with distribution $p$. Lemma~\ref{lemma:unimod_is_reversible} shows that $\widetilde{p}$ is a reversible stationary distribution for $M_T$.

List the eigenvalues of $M_T$ in decreasing order, $1=\lambda_1 \geq \lambda_2 \geq \ldots \geq \lambda_t$. We say $M_T$ is \emph{bipartite} if $\lambda_t=-1$.
$M_T$ is bipartite if and only if we can partition $T$ into two sets $T_1$ and $T_2$ such that whenever $o_i, o_j \in T_1$ or $o_i, o_j \in T_2$, we have $p_{M_T}(o_i,o_j)=0$. We set $\rho_T = \max(\{0\}\cup \{|\lambda_i| ~|~ 1<i\le t ~\textrm{ and }~\lambda_i>-1\})$.

We will have to treat the bipartite and non-bipartite case separately. When $M_T$ is not bipartite, we have $\rho_T = 0$ when $t=1$ and $\rho_T = \max\{\lambda_2,|\lambda_t|\}$ when $t\geq2$. The following is an immediate consequence of the Convergence Theorem for finite-state Markov chains.

\begin{lemma}\label{lemma:markov_chain_convergence}
Assume $M_T$ is not bipartite. Let $e_{o_i} \in \R^T$ denote the characteristic vector of $o_i\in T$. Then for any $v\in \mathbb{R}^T$, there exists a $C>0$ such that for any $i\in[t]$ and $k \in \N$ we have
\[
    |\langle M_T^k e_{o_i},  v \rangle - \langle \widetilde{p}, v\rangle|\le C\rho^k_T.
\]
\end{lemma}

When $M_T$ is bipartite, we have $\rho_T = 0$ when $t=2$, and as the spectrum is symmetric, in fact $\lambda_{t-1}=-\lambda_2$, so $\rho_T =\lambda_2$ whenever $t\geq3$.
The reversibility ensures $\widetilde{p}(T_1) = \widetilde{p}(T_2)=1/2$, and $M^2_T$ defines two disjoint Markov chains on $T_1$ and $T_2$ with stationary measures $2\widetilde{p}|_{T_1}$ and $2\widetilde{p}|_{T_2}$ respectively. The eigenvalues of $M^2_T$ are the squares of the eigenvalues of $M_T$, in particular they are non-negative, so  $M^2_T$ is not bipartite on either $T_1$ or $T_2$. Also the second largest eigenvalue in absolute value of $M_T^2$ (on both $T_1$ and $T_2$) is $\rho^2_T$.

\begin{remark}
The key to our proofs is not quasi-transitivity itself, but rather the exponential rate of convergence to the stationary distribution of $M_T$ as given by Lemma \ref{lemma:markov_chain_convergence}. So all our results hold more generally whenever the Markov chain on the orbit section is uniformly ergodic.
\end{remark}

\subsection{Graphings}\label{subsec:gphings}
Graphings play an essential role in obtaining invariant random structures on graphs as they represent a space where both the probability measure and the underlying (possibly random) countable graph are present. Their use in constructing factor-of-iid perfect matchings is well-established in \cite{LyonsNazarov} and \cite{csoka2017invariant}. For a more detailed introduction, see for example \cite[Chapter 18]{lovasz2012large}.

\begin{defi}\label{def:graphing}
Let $(X, \nu)$ be a Borel probability space.
A (bounded-degree) \emph{graphing} is a graph $\mathcal{G}$ with $V(\mathcal{G}) = X$ and Borel edge set $E(\mathcal{G})$, in which all degrees are at most $D \in \N$, and 
\begin{equation} \label{eqn:graphing}
\int_{A} \deg_B (x) \ d \nu (x) = \int_{B} \deg_A(x) \ d \nu (x)
\end{equation}
for all measurable sets $A,B \subseteq X$, where $\deg_S(x)$ is the number of edges from $x \in X$ to $S \subseteq X$.
\end{defi}

We will now define what we mean by $E(\mathcal{G})$ being Borel. The reason we do it in a slightly convoluted way is because in this paper it will be more convenient to use $E(\mathcal{G})$ to denote the set of edges, and not think about it as a symmetric subset of $X \times X$. The present downside to this is that defining the Borel structure and the edge measure can be done most naturally inside $X \times X$. For this reason let $\tilde{E}(\mathcal{G})$ denote the symmetric subset of $X \times X$ corresponding to the edges of $\mathcal{G}$:

\[\tilde{E}(\mathcal{G})=\{(x,y) \in X \times X ~|~ xy \in E(\mathcal{G})\}.\]

We say $E(\mathcal{G})$ is a Borel edge set if $\tilde{E}(\mathcal{G}) \subseteq X \times X$ is Borel. Then $E(\mathcal{G})$ itself has a Borel $\sigma$-algebra corresponding to the sub-$\sigma$-algebra of symmetric Borel subsets of $\tilde{E}(\mathcal{G})$. 

Moreover, the measure $\nu$ of a graphing $\mathcal{G}$ gives rise to a measure $\nu_{\tilde{E}}$ on $X\times X$ by defining
\[\nu_{\tilde{E}}(A\times B)=\frac{1}{2}\int_{A}\deg_B(x)\ d \nu(x)\]
for any measurable $A,B\subset X$. The measure $\nu_{\tilde{E}}$ is concentrated on $\tilde{E}(\mathcal{G})$. 

We then define the \emph{edge measure} $\nu_E$ on $E(\mathcal{G})$ by setting $\nu_E(F) = \nu_{\tilde{E}}(\tilde{F})$ for measurable subsets $F \subseteq E(\mathcal{G})$. ($\tilde{F}$ is defined analogously to $\tilde{E}(\mathcal{G})$.) Essentially we are restricting $\nu_{\tilde{E}}$ as defined above to the symmetric Borel subsets of $\tilde{E}(\mathcal{G})$. 
The factor $1/2$ is introduced so that the appropriate version of the usual edge double counting identity $2|E(G)|=\sum_{v\in V(G)} \deg(v)$ for finite graphs also holds for graphings:
\[2\nu_E\big(E(\mathcal{G})\big)=\int_{X}\deg(x)\ d \nu(x).\]

\begin{example} \label{exmpl:sch_graphing}
Given a finitely generated group $\Gamma=\langle S\rangle$ and a p.m.p.\ action $\Gamma \acts (X,\nu)$, the Schreier graph $\Sch(\Gamma \acts X, S)$ is a graphing (after forgetting the orientation and $S$-labelling). The action being p.m.p.\ implies that the degree condition (\ref{eqn:graphing}) holds.
\end{example}

\subsection{Connection to Bernoulli graphings}\label{subsec:bernoulli_graphing}
We now introduce Bernoulli graphings, which are closely related to factors of iid.

For a unimodular quasi-transitive graph $G$, we define its Bernoulli graphing $\mathcal G$ as follows. The vertex set of $\mathcal G$ is $\Omega$, the space of $[0,1]$-decorated, rooted, connected graphs with degree bound $D$ (up to rooted isomorphism). Elements of $V(\mathcal G)=\Omega$ are of the form $(H,u,\omega)$, where $(H,u)$ is a connected, bounded-degree rooted graph, and $\omega:V(H) \to [0,1]$ is a labelling. We connect $(H,u,\omega)$ with $(H',u',\omega')$ if and only if we can obtain $(H',u',\omega')$ from $(H,u,\omega)$ by moving the root $u$ to one of its neighbors. We denote the resulting measurable edge set by $\mathcal E$.

It remains to define the probability measure on $\Omega$. (Note that the vertex and edge sets are the same for every $G$, only the measure will be different.) $G$ is quasi-transitive, so it has finitely many possible rooted versions, namely the $(G,o_i)$ for $o_i \in \{o_1, \ldots, o_t\}$. Let us pick the rooted graph $(G,o_i)$ with probability $p(o_i)$. 
Then pick a labelling $\omega\in[0,1]^{V(G)}$ according to ${\tt u}^{V(G)}$. Recall that ${\tt u}$ stands for the uniform measure on $[0,1]$. Let $\nu_{G}$ denote the distribution of $(G,o,\omega)$. Then $\nu_G$ is a probability measure on $\Omega$. The Bernoulli graphing of $G$ is $\mathcal G=(\Omega, \mathcal E, \nu_{G})$. $\mathcal G$ satisfies (\ref{eqn:graphing}) because $G$ is unimodular.

Given a unimodular quasi-transitive graph $G$, constructing an $\Aut(G)$-factor of iid balanced orientation (Schreier decoration) of $G$ is equivalent to constructing a measurable balanced orientation (Schreier decoration) of the Bernoulli graphing $\mathcal{G}$ built on $G$. Here, measurability means that the oriented edges (and the color classes) form measurable subsets of $X \times X$.

Also note that a measurable Schreier decoration of any graphing $\mathcal{G}$ defines a p.m.p.\ action $F_d \acts V(\mathcal{G})$ that \emph{generates} the graphing as in Example~\ref{exmpl:sch_graphing}.

Therefore, an equivalent formulation of our main motivating question is the following: given a (quasi-transitive unimodular) $2d$-regular graph $G$, is the Bernoulli graphing $\mathcal{G}$ generated by a p.m.p.\ action of $F_d$? It would of course be even better to answer this question for all unimodular random rooted graphs.

\subsection{Perfect matchings in expanding graphings}
Finding measurable perfect matchings in graphings is usually achieved through expansion properties. We will use the following two results, both based on the argument of Lyons and Nazarov in \cite{LyonsNazarov}.

\begin{theorem}[Lyons-Nazarov, \cite{LyonsNazarov}] \label{thm:graphing_perfect_matching}
Let $\mathcal{G}=(X, E, \nu)$ be a graphing with no odd cycles. Assume it has vertex expansion at least $c > 1$. That is, for any $A \subset X$ such that $0<\nu(A) \le 1/2$, we have
\[ \frac{\nu(N_{\mathcal{G}}(A))}{\nu(A)} \ge c.\]
Then $\mathcal{G}$ has a Borel matching that covers all vertices up to a nullset.
\end{theorem}
We will also need a  variation for measurably bipartite graphings.

\begin{theorem}[Lyons-Nazarov, Theorem 9.1 in \cite{pikhurko2020borel}] \label{thm:bipartite_graphing_perfect_matching}
Let $\varepsilon>0$. Let $\mathcal{G}=(X_1, X_2,E, \nu)$ be a measurably bipartite graphing with $\nu(X_1) = \nu(X_2)$. Assume it has bipartite vertex-expansion at least $1+\varepsilon$. That is, for any $A \subseteq X_1$ and $B \subseteq X_2$, we have
\[ \nu(N_{\mathcal{G}}(A)) \ge \min\left\{(1+\varepsilon)\nu(A),\frac{1}{4}+\varepsilon\right\} \qquad\textrm{ and }\qquad  \nu(N_{\mathcal{G}}(B))\ge \min\left\{(1+\varepsilon)\nu(B),\frac{1}{4}+\varepsilon\right\}.\]
Then $\mathcal{G}$ has a Borel matching that covers all vertices up to a nullset.
\end{theorem}

\section{Spectral gap for non-amenable quasi-transitive graphs}
\label{section:spectral_gap}

In this section we prove our spectral theoretic result, Theorem~\ref{thm:spectral_gap}.

For a graphing $\mathcal{G}$ on $(X, \nu)$, we denote by $\nu_{\mathrm{st}}$ the degree-biased version of $\nu$, that is 

\[\nu_{\mathrm{st}}(A) = \int_{x \in A} \deg(x) d \nu \bigg/ \int_{x \in X} \deg(x) d \nu. \]

As the notation suggests, $\nu_{\mathrm{st}}$ is stationary with respect to the Markov operator $\mathcal{M}$ of $\mathcal{G}$ that is defined by
\[
    (\mathcal{M} f)(G,o,\omega)=\frac{1}{\deg_G(o)}\sum_{ov\in E} f(G,v,\omega).
\]

$\mathcal{M}$ is a self-adjoint operator on $L^2(\Omega,\nu_{\mathrm{st}})$. To get a bound on the spectral radius of $\mathcal{M}$ (on the appropriate subspace), we will use the following lemma.

\begin{lemma}[Lemma~2.4 of \cite{backhausz2015ramanujan}]\label{lemma:corsp} For a bounded self-adjoint operator $T$ on a Hilbert space $\mathcal{H}$ and for any spanning subset $H$ of $\mathcal{H}$, we have

\[
    \rho(T)=\|T\|=\sup_{v\in H}\left(\limsup_{k\to\infty}\left|\frac{\langle v,T^k v \rangle}{\langle v,v\rangle}\right|^{1/k}\right).
\]
\end{lemma}

The following two theorems deal with the non-bipartite and bipartite case separately. Recall that we denote the Markov operator of $G$ on $\ell^2(G, m_{\mathrm{st}})$ by $M$, where $m_{\mathrm{st}}$ denotes the degree-biased version of the counting measure on $V(G)$. As $G$ is non-amenable we have $\rho=||M||<1$. Recall also that $\rho_T = \max(\{0\}\cup \{|\lambda_i| ~|~ 1<i\le t ~\textrm{ and }~\lambda_i>-1\})$ is defined through the finite state Markov chain $M_T$, and  $\rho_T<1$.

\begin{theorem} \label{thm:non_bipart_spectral_gap}
Let $G$ be as in Theorem $\ref{thm:spectral_gap}$, and assume also that $M_T$ is not bipartite. Let $L_{0}^2(\Omega,\nu_{\mathrm{st}})$ denote the orthogonal complement of the subspace of constant functions. Then the spectral radius of $\mathcal{M}$ on $L_0^2(\Omega,\nu_{\mathrm{st}})$ is at most $\max\{\rho, \rho_T\}<1$.
\end{theorem}

\begin{theorem} \label{thm:bipart_spectral_gap}
Let $G$ be as in Theorem $\ref{thm:spectral_gap}$, and assume that $M_T$ is bipartite. Let $\rho < 1$ denote the spectral radius of $G$ on $\ell^2(G, m_{\mathrm{st}})$. The Bernoulli graphing $\mathcal{G}$ is measurably bipartite, with bipartition $X_1 \cup X_2 = V(\mathcal{G})$. Let $L_{00}^2(\Omega,\nu_{\mathrm{st}})$ denote the orthogonal complement of the subspace generated by the functions $\mathbbm{1}_X$ and $\mathbbm{1}_{X_1}-\mathbbm{1}_{X_2}$. Then the spectral radius of $\mathcal{M}$ on $L_{00}^2(\Omega,\nu_{\mathrm{st}})$ is at most $\max\{\rho, \rho_T\}<1$.
\end{theorem}

\begin{proof}[Proof of Theorem~\ref{thm:spectral_gap}]
The content of Theorem~\ref{thm:spectral_gap} is exactly Theorems~\ref{thm:non_bipart_spectral_gap} and~\ref{thm:bipart_spectral_gap}.
\end{proof}

We first prove Theorem~\ref{thm:non_bipart_spectral_gap} and then use it to prove Theorem~\ref{thm:bipart_spectral_gap}.

\begin{proof}[Proof of Theorem~\ref{thm:non_bipart_spectral_gap}]
As before, let $p_k(o,y)$ denote the probability that the 
random walk on $G$ starting at $o$ arrives at $y$ after $k$ steps. We have $\limsup_{k \to \infty} \big(p_k(o,y)\big)^{1/k} = \rho$, so for every $\varepsilon > 0$ there exists some $C_0(o,y,\varepsilon) \in \R$ such that $p_k(o,y) \le C_0(o,y,\varepsilon)(\rho + \varepsilon)^k$ for all $k$.

During this proof, we will write $\mu$ for ${\tt u}^{V(G)}$. We will use Lemma~\ref{lemma:corsp} in the following setting.
Let $H\subseteq L^2_0(\Omega,\nu_{\mathrm{st}})$ be the set of functions $f$ such that
\begin{itemize}
    \item $f$ has zero mean, i.e.
    \[
        \int_{(G,o,\omega)\in \Omega}f(G,o,\omega)~d\nu_{\mathrm{st}}=\sum_{i=1}^t \widetilde{p}(o_i) \int_{\omega\in[0,1]^{V(G)}} f(G,o_i,\omega)~d\mu=0;
    \]
    \item $f$ has norm 1, i.e.
    \[
        \int_{(G,o,\omega)\in \Omega}f^2(G,o,\omega)~d\nu_{\mathrm{st}} =\sum_{i=1}^t \widetilde{p}(o_i) \int_{\omega\in[0,1]^{V(G)}} f^2(G,o_i,\omega)~d\mu=1;
    \]
    \item there exists some $r\ge 0$ such that if we change labels of vertices further than $r$ from the root then the value of $f$ doesn't change.
\end{itemize}

The set $H$ is a spanning subset of $L^2_0(\Omega,\nu_{\mathrm{st}})$.
(Note that a measurable function $f: \Omega \to \R$ defines an $\R$-valued factor of iid on any graph $G$. Indeed, for $\omega \in [0,1]^{V(G)}$ one defines $\big(\Phi(\omega)\big)(v) = f(G,v,\omega)$.)
This is equivalent to saying that any factor of iid process is a limit of block factors; see \cite{lyons2017factors}.

Let us fix an element $f\in H$. Then
\begin{align*}
    \langle\mathcal{M}^k f,f\rangle&=\sum_{i=1}^t \widetilde{p}(o_i) \int_{\omega\in [0,1]^{V(G)}} \sum_{y\in B_{k}(o_i)} p_k(o_i,y) f(G,o_i,\omega)f(G,y,\omega)d\mu\\
    &=\sum_{i=1}^t \widetilde{p}(o_i)  \sum_{y\in B_{k}(o_i)} p_k(o_i,y) \int_{\omega\in [0,1]^{V(G)}} f(G,o_i,\omega)f(G,y,\omega)d\mu.
\end{align*}

We split the sum depending on the distance between $o_i$ and $y$:
\begin{align} \label{eqn:far_term}
    \langle\mathcal{M}^k f,f\rangle&=\sum_{i=1}^t \widetilde{p}(o_i)  \sum_{y\notin B_{2r}(o_i)} p_k(o_i,y) \int_{\omega\in [0,1]^{V(G)}} f(G,o_i,\omega)f(G,y,\omega)d\mu\\ \label{eqn:close_term}
    &+\sum_{i=1}^t \widetilde{p}(o_i) \sum_{y\in B_{2r}(o_i)} p_k(o_i,y) \int_{\omega\in [0,1]^{V(G)}}  f(G,o_i,\omega)f(G,y,\omega)d\mu.
\end{align}

If the distance between $o_i$ and $y$ is bigger than $2r$ then (by the third property of $f$) the values $f(G,o_i,\omega)$ and $f(G,y,\omega)$ depend on labels at disjoint sets of vertices. Since those labels are independent, we have
\[
    \int_{\omega\in [0,1]^{V(G)}} f(G,o_i,\omega)f(G,y,\omega)d\mu=\int_{\omega\in [0,1]^{V(G)}} f(G,o_i,\omega)d\mu \int_{\omega\in [0,1]^{V(G)}}f(G,y,\omega)d\mu.
\]
Therefore the first term, $(\ref{eqn:far_term})$ is 
\begin{align}
    \sum_{i=1}^t \widetilde{p}(o_i)&  \sum_{y\notin B_{2r}(o_i)} p_k(o_i,y) \int_{\omega\in [0,1]^{V(G)}} f(G,o_i,\omega)f(G,y,\omega)d\mu \nonumber\\
    =& \sum_{i=1}^t \widetilde{p}(o_i)  \sum_{y\notin B_{2r}(o_i)} p_k(o_i,y) \int_{\omega\in [0,1]^{V(G)}} f(G,o_i,\omega)d\mu \int_{\omega\in [0,1]^{V(G)}}f(G,y,\omega)d\mu \nonumber\\
    =&\sum_{i=1}^t \widetilde{p}(o_i)  \int_{\omega\in [0,1]^{V(G)}} f(G,o_i,\omega)d\mu \sum_{y\notin B_{2r}(o_i)} p_k(o_i,y)  \int_{\omega\in [0,1]^{V(G)}}f(G,y,\omega)d\mu \nonumber\\
    =&\sum_{i=1}^t \widetilde{p}(o_i)  \int_{\omega\in [0,1]^{V(G)}} f(G,o_i,\omega)d\mu \sum_{y\in V(G)} p_k(o_i,y)  \int_{\omega\in [0,1]^{V(G)}}f(G,y,\omega)d\mu \nonumber\\
    &-\sum_{i=1}^t \widetilde{p}(o_i)  \int_{\omega\in [0,1]^{V(G)}} f(G,o_i,\omega)d\mu \sum_{y\in B_{2r}(o_i)} p_k(o_i,y)  \int_{\omega\in [0,1]^{V(G)}}f(G,y,\omega)d\mu \nonumber\\
    \label{eqn:markov_term}
    =&\sum_{i=1}^t \widetilde{p}(o_i)  \int_{\omega\in [0,1]^{V(G)}} f(G,o_i,\omega)d\mu \sum_{j=1}^t p^{M_T}_k(o_i,o_j)  \int_{\omega\in [0,1]^{V(G)}}f(G,o_j,\omega)d\mu\\
    \label{eqn:indep_close_term}
    &-\sum_{i=1}^t \widetilde{p}(o_i)  \int_{\omega\in [0,1]^{V(G)}} f(G,o_i,\omega)d\mu \sum_{y\in B_{2r}(o_i)} p_k(o_i,y)  \int_{\omega\in [0,1]^{V(G)}}f(G,y,\omega)d\mu.
\end{align}

Along the calculation, we used that $\int_{\omega\in [0,1]^{V(G)}}f(G,y,\omega) \ d \mu$ only depends on the orbit that $y$ is in. Then we grouped all the $y \in \Gamma.{o_j}$ together, and used the fact that 

\[\sum_{y \in \Gamma.o_j} p_k(o_i,y) = p^{M_T}_k(o_i,o_j).\]
Indeed, the probability of the
random walk on $G$ started from $o_i$ ending up at some $y \in \Gamma.o_j$ after $k$ steps is
the same as the probability of the
finite Markov chain $M_T$, starting from $o_i$ ending up in $o_j$ after $k$ steps.

We now use that from any initial state, the finite Markov chain converges to the stationary distribution. That is, we use Lemma~\ref{lemma:markov_chain_convergence}, with the vector 

\[v: o_j \mapsto \int_{\omega\in [0,1]^{V(G)}}f(G,o_j,\omega)d\mu.\]
We get that there exists some $C_1 \in \mathbb{R}$ such that 

\[\left|\sum_{j=1}^t p^{M_T}_k(o_i,o_j)  \int_{\omega\in [0,1]^{V(G)}}f(G,o_j,\omega)d\mu - \sum_{j=1}^t \widetilde{p}(o_j) \int_{\omega\in[0,1]^{V(G)}} f(G,o_j,\omega)~d\mu \right| \leq C_1 \rho^k_T.\]
Note that $C_1$ might depend on $f$, but not on $k$. The first property of $f$ says the second term in the absolute value is 0, so we have

\[\left|\sum_{j=1}^t p^{M_T}_k(o_i,o_j)  \int_{\omega\in [0,1]^{V(G)}}f(G,o_j,\omega)d\mu \right| \leq C_1 \rho^k_T.\]
We use this to bound the term (\ref{eqn:markov_term}):
\begin{align*}
    \left|\sum_{i=1}^t \widetilde{p}(o_i)  \int_{\omega\in [0,1]^{V(G)}} f(G,o_i,\omega)d\mu \sum_{j=1}^t p^{M_T}_k(o_i,o_j)  \int_{\omega\in [0,1]^{V(G)}}f(G,o_j,\omega)d\mu\right|\\
    \le \sum_{i=1}^t \widetilde{p}(o_i)  \left|\int_{\omega\in [0,1]^{V(G)}} f(G,o_i,\omega)d\mu\right| \left|\sum_{j=1}^t p^{M_T}_k(o_i,o_j)  \int_{\omega\in [0,1]^{V(G)}}f(G,o_j,\omega)d\mu\right|\\
    \le\sum_{i=1}^t \widetilde{p}(o_i)  \left|\int_{\omega\in [0,1]^{V(G)}} f(G,o_i,\omega)d\mu\right| C_1\rho^k_T=C_2\rho_T^k.
\end{align*}

To recap, we had $\langle\mathcal{M}^k f,f\rangle = (\ref{eqn:far_term}) + (\ref{eqn:close_term}) = (\ref{eqn:markov_term}) - (\ref{eqn:indep_close_term}) + (\ref{eqn:close_term}).$ We have already bounded the absolute value of $(\ref{eqn:markov_term})$, so we now bound the absolute values of $(\ref{eqn:close_term})$ and $(\ref{eqn:indep_close_term})$. These terms, however, correspond to cases where the random walk on $G$ arrives close to the starting point after $k$ steps. As $G$ is non-amenable, the probability of this happening decays exponentially in $k$. 

Formally, let us recall that $p_k(o_i,y)\le C_0(o_i, y,\varepsilon)(\rho+\varepsilon)^k$. We write 

\begin{align*}
    \big|(\ref{eqn:close_term})\big| &= \left|\sum_{i=1}^t \widetilde{p}(o_i) \sum_{y\in B_{2r}(o_i)} p_k(o_i,y) \int_{\omega\in [0,1]^{V(G)}}  f(G,o_i,\omega)f(G,y,\omega)d\mu\right| \leq \\
    & \leq (\rho+\varepsilon)^k \underbrace{\sum_{i=1}^t \widetilde{p}(o_i) \sum_{y\in B_{2r}(o_i)} C_0(o_i,y,\varepsilon) \left|\int_{\omega\in [0,1]^{V(G)}}  f(G,o_i,\omega)f(G,y,\omega)d\mu\right|}_{C_3},
\end{align*}

\begin{align*}
    \big|(\ref{eqn:indep_close_term})\big| &= \left|\sum_{i=1}^t \widetilde{p}(o_i) \sum_{y\in B_{2r}(o_i)} p_k(o_i,y) \int_{\omega\in [0,1]^{V(G)}}f(G,o_i,\omega)d\mu\int_{\omega\in [0,1]^{V(G)}}f(G,y,\omega)d\mu \right| \leq \\
    & \leq (\rho+\varepsilon)^k \underbrace{\sum_{i=1}^t \widetilde{p}(o_i) \sum_{y\in B_{2r}(o_i)} C_0(o_i,y,\varepsilon) \left|\int_{\omega\in [0,1]^{V(G)}}f(G,o_i,\omega)d\mu\int_{\omega\in [0,1]^{V(G)}}f(G,y,\omega)d\mu\right|}_{C_4}.
\end{align*}

Note that the constants $C_3$ and $C_4$ depend on $f$, but not on $k$. We now combine our bounds and get
\begin{align*}
    \limsup_{k\to \infty} \left|\langle \mathcal{M}^k f,f\rangle \right|^{1/k} &\le \limsup_{k\to \infty} \big(C_2\rho_T^k+(C_3+C_4)(\rho+\varepsilon)^k\big)^{1/k} \\
    &= \lim_{k\to \infty} \big(C_2\rho_T^k+(C_3+C_4)(\rho+\varepsilon)^k\big)^{1/k} = \max(\rho_T,\rho+\varepsilon).
\end{align*}

This holds for any $\varepsilon >0$, so we have $\limsup_{k\to \infty} \left|\langle \mathcal{M}^k f,f\rangle \right|^{1/k} \leq \max(\rho_T, \rho)$. By Lemma~\ref{lemma:corsp}, we now have $||\mathcal{M}|_{L^2_0(\Omega,\nu_{\mathrm{st}})}|| \leq \max(\rho_T,\rho)$, which completes the proof.
\end{proof}

\begin{proof}[Proof of Theorem~\ref{thm:bipart_spectral_gap}]
Note that by bipartiteness, the subspaces $\1_{X_1}^\perp$ and $\1_{X_2}^\perp$ of $L^2(\Omega)$ are invariant under the action of $\mathcal{M}^2$, so
$\mathcal{M}^2$ is well-defined as an operator on $S_1=\{f\in L^2_{00}(\Omega,\nu_{\mathrm{st}}) : f|_{X_2}\equiv0\}$ and on $S_2=\{f\in L^2_{00}(\Omega,\nu_{\mathrm{st}}) : f|_{X_1}\equiv0\}$.
Moreover, $L^2_{00}(\Omega,\nu_{\mathrm{st}})$ can be written as the internal direct sum
\[L_{00}^2(\Omega, \nu_{\mathrm{st}}) = S_1 \oplus S_2,\]
so the spectrum of $\mathcal{M}^2|_{L_{00}^2(\Omega, \nu_{\mathrm{st}})}$ satisfies
\[
\sigma\left(\mathcal{M}^2|_{L_{00}^2(\Omega, \nu_{\mathrm{st}})}\right) = \sigma\left(\mathcal{M}^2|_{S_1}\right) \cup \sigma\left(\mathcal{M}^2|_{S_2}\right).
\]

$M^2_T$ restricted to $T_1$ or $T_2$ 
is not bipartite, so the proof of Theorem~\ref{thm:non_bipart_spectral_gap}, applied to $\mathcal{M}^2$ on
$S_1$ and on $S_2$,
yields that the two spectral radii are both at most $\max(\rho^2, \rho_T^2)$.

Finally, we have $\sigma(\mathcal{M}^2|_{L_{00}^2}) = \{\lambda^2 \mid \lambda \in \sigma(\mathcal{M}_{L_{00}^2})\}$, which completes the proof.
\end{proof}

\section{Perfect matchings in quasi-transitive graphs}\label{section:perfect_matchings}
In this section, we prove Corollary~\ref{cor:quasi_transitive_matching} using our spectral-theoretic results. We will obtain the necessary expansion properties from the spectral gap through the following two lemmas.

\begin{lemma}[\cite{LyonsNazarov}, Lemma 2.3] \label{lemma:LyonsNazarov_spectral_gap_implies_expansion}
Let $\mathcal{G}=(X,E, \nu)$ be a graphing, and let $\rho_{\mathcal{G}} = \rho\left(\mathcal{M}|_{L_{0}^2(X,\nu_{\textrm{st}})}\right)$. Let $B \subseteq X$ be a measurable subset, and let $b= \nu_{\textrm{st}}(B)/\nu_{\textrm{st}}(X)$ denote the degree-biased density of $B$ in $X$. Let $b'=\nu_{\textrm{st}}(N(B))/\nu_{\textrm{st}}(X)$ denote the degree-biased density of the neighbors of $B$ in $X$. Then
\[b' \geq \frac{1}{\rho_{\mathcal{G}}^2 (1-b) + b} \cdot b.\]
\end{lemma}

We also show a version of Lemma~\ref{lemma:LyonsNazarov_spectral_gap_implies_expansion} for measurably bipartite graphings.

\begin{lemma} \label{lemma:measurably_bipart_expansion_withoutregularity}
Let $\mathcal{G}=(X_1,X_2,E, \nu)$ be a measurably bipartite graphing,
and let $\rho_{\mathcal{G}} = \rho\left(\mathcal{M}|_{L_{00}^2(X,\nu_{\textrm{st}})}\right)$. Let $B \subseteq X_1$ be a measurable subset, and let $b= \nu_{\textrm{st}}(B)/\nu_{\textrm{st}}(X_1)$ denote the degree-biased density of $B$ in $X_1$. Let $b'=\nu_{\textrm{st}}(N(B))/\nu_{\textrm{st}}(X_2)$ denote the degree-biased density of the neighbors of $B$ in $X_2$. Then
\[b' \geq \frac{1}{\rho_{\mathcal{G}}^2 (1-b) + b} \cdot b.\]
The same holds for measurable subsets $B \subseteq X_2$.
\end{lemma}

\begin{proof}[Proof, following Lemma 2.3 in \cite{LyonsNazarov}]
First we note that by the graphing condition (\ref{eqn:graphing}), we must have $\nu_{\textrm{st}}(X_1)=\nu_{\textrm{st}}(X_2)=\frac{1}{2}$:
\[
\nu_{\textrm{st}}(X_1)=\frac{\int_{X_1}\deg(x) d\nu}{\int_X\deg(x) d\nu}=\frac{\int_{X_1}\deg_{X_2}(x) d\nu}{\int_X\deg(x) d\nu}=\frac{\int_{X_2}\deg_{X_1}(x) d\nu}{\int_X\deg(x) d\nu}=\frac{\int_{X_2}\deg(x) d\nu}{\int_X\deg(x) d\nu}=\nu_{\textrm{st}}(X_2).
\]
Since $\mathcal{M} \1_B$ is constant 0 on the complement of $B'=N(B)$, we have
\[ \nu_{\textrm{st}}(B) = \langle \1_B, \1 \rangle = \langle \1_B, \mathcal{M} \1  \rangle = \langle \mathcal{M} \1_B, \1 \rangle = \langle \mathcal{M} \1_B, \1_{B'} \rangle.\]
Consequently,
\begin{equation} \label{eqn:CSB}
\nu_{\textrm{st}}(B)^2=\langle \mathcal{M} \1_B, \1_{B'} \rangle^2 \leq ||\mathcal{M} \1_B||^2 \cdot ||\1_{B'}||^2 = ||\mathcal{M} \1_B||^2 \cdot \nu_{\textrm{st}}(B') = ||\mathcal{M} \1_B||^2 \cdot \frac{b'}{2}.
\end{equation}

We split $\1_B$ as follows: $\1_B = b\1_{X_1} + f_B$, where $f_B = \1_B - b\1_{X_1} = (1-b)\1_{B} + (-b) \1_{X_1 \setminus B}$. Notice that $f_B \perp \1$ and $f_B \perp \1_{X_1} - \1_{X_2}$, therefore $||\mathcal{M} f_B|| \leq \rho_{\mathcal{G}} \cdot ||f_B||$. Moreover,
\[||f_B||^2= (1-b)^2 \cdot \nu_{\textrm{st}}(B) + b^2 \cdot \nu_{\textrm{st}}(X_1 \setminus B) = (1-b)^2 \cdot \frac{b}{2} + b^2 \cdot \frac{1-b}{2} = \frac{b(1-b)}{2}.\]

Now $\mathcal{M}\1_B=b\cdot \mathcal{M}\1_{X_1} + \mathcal{M}f_B= b\1_{X_2} + \mathcal{M}f_B$. Again, $\1_{X_2} \perp \mathcal{M}f_B$ because $\langle \1_{X_2}, \mathcal{M}f_B \rangle = \langle \mathcal{M} \1_{X_2}, f_B \rangle = \langle \1_{X_1}, f_B \rangle = 0$. Hence,
\begin{equation} \label{eqn:P1_B_square}
||\mathcal{M} \1_B||^2 = b^2 ||\1_{X_2}||^2 + ||\mathcal{M} f_B||^2 \leq b^2 \cdot \nu(X_2) + \rho_{\mathcal{G}}^2 \cdot ||f_B||^2 = \frac{1}{2}\left( b^2 + \rho_{\mathcal{G}}^2 b (1-b)\right).
\end{equation}

Putting (\ref{eqn:CSB}) and (\ref{eqn:P1_B_square}) together, we get
\[b' \geq \frac{2\nu_{\textrm{st}}(B)^2}{||\mathcal{M} \1_B||^2}=\frac{b^2}{2||\mathcal{M} \1_B||^2} \geq \frac{b^2}{b^2 + \rho_{\mathcal{G}}^2 b(1-b)} = \frac{1}{\rho_{\mathcal{G}}^2 (1-b) + b} \cdot b.\]
\end{proof}

\begin{proof}[Proof of Corollary~\ref{cor:quasi_transitive_matching}]
Let $M_T$ denote the finite state Markov chain defined by the quasi-transitive graph $G$ described in subsection~\ref{subsec:unimod_qt_graphs}. If $M_T$ is not bipartite, we have spectral gap on $L_0^2(V(\mathcal{G}),\nu)$ by Theorem~\ref{thm:non_bipart_spectral_gap}, which implies vertex expansion by Lemma~\ref{lemma:LyonsNazarov_spectral_gap_implies_expansion}, and Theorem~\ref{thm:graphing_perfect_matching} provides the perfect matching. 
If $M_T$ is bipartite, the Bernoulli graphing $\mathcal{G}$ is measurably bipartite and has spectral gap by Theorem~\ref{thm:bipart_spectral_gap}. This implies bipartite expansion by Lemma~\ref{lemma:measurably_bipart_expansion_withoutregularity}. The bipartite expansion implies the existence of a perfect matching by Theorem~\ref{thm:bipartite_graphing_perfect_matching}.

Note that we use the regularity of $G$, as it implies that the probability measures $\nu$ (used in Theorems~\ref{thm:graphing_perfect_matching} and \ref{thm:bipartite_graphing_perfect_matching}) and $\nu_{\textrm{st}}$ (used in Lemmas~\ref{lemma:LyonsNazarov_spectral_gap_implies_expansion} and \ref{lemma:measurably_bipart_expansion_withoutregularity}) coincide.
\end{proof}

\begin{remark} \label{rmk:quasi_trans_no_perfect_matching}
Abért, Csóka, Lippner and Terpai show in \cite{csoka2017invariant} that any infinite transitive graph has a perfect matching. The following example shows that this is not true for quasi-transitive graphs. Therefore if we want to extend the result of Lyons and Nazarov on factor-of-iid perfect matchings beyond transitive graphs, assuming $G$ to be bipartite is necessary.

Let $G$ be any unimodular transitive non-amenable $2d$-regular graph, e.g.\ the tree $T_{2d}$. Let us now attach two pendant $K_{2d+5}^-$ (the complete graph minus an edge) to every vertex of $G$ so that the resulting graph $\widetilde{G}$ is $2d+4$-regular and has three orbits. $\widetilde{G}$ is quasi-isometric to $G$, and so it is non-amenable. To see that it is unimodular as well, we refer to \cite{beringer2017}, where it is shown that performing certain local changes preserves unimodularity. Every vertex $v$ in $\widetilde{G}$ corresponding to an original vertex in $G$ is now a cut vertex, and at least two of the components left in $\widetilde{G}$ when $v$ is removed are finite and having odd order. $\widetilde{G}$ has therefore no perfect matching at all, let alone a factor-of-iid one.
\end{remark}

\section{Balanced orientations} \label{sec:balanced_orientation}
In this subsection, we prove Theorem~\ref{thm:non-amenable_balanced_orientation}. We will use an auxiliary bipartite graph $G^*$ whose perfect matchings correspond to balanced orientations of our graph $G$. 
This connection is implicit in Schrijver's paper about counting eulerian orientations \cite{schrijver1983bounds}, and first explicit constructions of pairs of graphs in which a balanced orientation of one is a perfect matching of the other were given by Mihail and Winkler \cite{MihailWinkler}.
The auxiliary graph $G^*$ is constructed from $G$ by local transformations, which makes sure that it is
quasi-isometric to $G$.

\begin{defi}
Let $G$ be a simple graph in which every vertex has an even degree. Then we define a simple graph $G^*$ as follows. $G^*$ has a vertex for every edge $e\in E(G)$ and $\deg(v)/2$ vertices for every vertex $v\in V(G)$, i.e.
\[
V(G^*)=\{x_e \colon e\in E(G)\}\cup\{v_i \colon v\in V(G), i\in[\deg(v)/2]\}.
\]
Then every vertex corresponding to a former edge is joined to all copies of its former endpoints.
\[
E(G^*)=\{x_{uv}v_i \colon uv\in E(G), i\in[\deg(v)/2]\}.
\]
\end{defi}

The vertices $x_e \in V(G^*)$ are called \emph{edge-type} vertices of $G^*$, and $v_i \in V(G^*)$ are called \emph{vertex-type} vertices. Any perfect matching $M$ in $G^*$ then defines a balanced orientation of $G$ by orienting an edge $e \in E(G)$ towards its endpoint $v$ if and only if $x_e$ and $v_i$ are matched by $M$ for some $i \in [\deg(v)/2]$.

We now introduce the same construction starting from a graphing $\mathcal{G}$. Recall that for a graphing $(\mathcal{G},\nu)$ we denote by $\nu_E$ the edge measure on $E(\mathcal{G})$.

\begin{defi} \label{def:aux_graphing}
Let $(\mathcal{G},\nu)$ be a graphing 
with finite average degree $\overline{\deg}=2\nu_E(E(\mathcal{G}))<\infty$ in which almost every vertex has even degree. Then we define the measurably bipartite auxiliary graphing $(\mathcal{G}^*,\nu^*)$ as follows.
\begin{itemize}
    \item $V(\mathcal{G}^*)=X_1\cup X_2$, where $X_1=E(\mathcal{G})$ and $X_2=\bigcup_{i=1}^\infty Y_i\times\{i\}$, where $Y_i=\{x\in V(\mathcal{G}) ~|~ \deg (x)\ge 2i\}$. Let us denote by $\pi: X_2\to V(\mathcal{G})$ the projection onto the first coordinate.
    \item The measure $\nu^*$ is defined by \[\nu^*|_{X_1}=\frac{1}{2\nu_E\big(E(\mathcal{G})\big)}\nu_E, \quad \nu^*|_{Y_i\times \{i\}} = \frac{1}{\int_{V(\mathcal{G})}\deg(v)d\nu(v)}\nu|_{Y_i}.\] 
    \item For $e \in X_1, x \in X_2$ there is an edge $ex \in E(\mathcal{G}^*)$ connecting them if and only if $\pi(x)\in e$.
\end{itemize}
\end{defi}
To check that $\mathcal{G}^*$ is indeed a graphing, we compute for any $A\subseteq X_1$ and $B\subseteq X_2$ that
\begin{align*}
    \int_{B}\deg_A(v)d\nu^*(v) &= \frac{\int_{V(\mathcal{G})} \big|\pi^{-1}(v) \cap B \big| \cdot \big|\{a\in A ~|~ v \textrm{ is incident to } a \}\big| \ d\nu (v)}{\int_{V(\mathcal{G})} \deg(u) d\nu (u)}=\\
 &= \frac{\int_A|\pi^{-1}(u)\cap B|+|\pi^{-1}(v)\cap B|d\nu_E(uv)}{2\nu_E(E(\mathcal{G}))}=\int_A\deg_B(e)d\nu^*(e).
\end{align*}

As in the discrete case, a measurable matching $M \subseteq E(\mathcal{G}^*)$ defines a measurable balanced orientation of $\mathcal{G}$ by orienting an edge $e \in E(\mathcal{G})$ towards its endpoint $v$ if $e$ and $(v,i)$ are matched by $M$ for some $i\in [\deg(v)/2]$.

We now go on to relate expansion properties of $\mathcal{G}$ to those of $\mathcal{G}^*$.
Let us define the \emph{Cheeger constant} of $\mathcal{G}$ as
\[
\Phi_{\textrm{st}}=\inf\left\{\frac{\int_S\deg_{N_{\mathcal{G}}(S)\setminus S}(u) d\nu(u)}{\nu_{\textrm{st}}(S)} ~\Big|~ 0<\nu_{\textrm{st}}(S)\leq\frac{1}{2}\right\}.
\]
Note that in this degree-biased version, we may have $\Phi_{\textrm{st}}>0$ even when the set of isolated vertices has positive $\nu$-measure.

\begin{lemma} \label{lemma:cheeger_implies_aux_bipart_expansion}
Let $(\mathcal{G},\nu)$ be a graphing with bounded average degree $\overline{\deg}<\infty$ and Cheeger constant $\Phi_{\textrm{st}}(\mathcal{G}) >0$. Then $(\mathcal{G}^*,\nu^*)$ has bipartite expansion, that is, there is an $\varepsilon>0$ such that for any $A\subseteq X_1$ and $B\subseteq X_2$, we have
\[
    \nu^*\big(N_{\mathcal{G}^*}(A)\big) \ge \min\left\{\left(1+\varepsilon\right)\nu^*(A),\frac{1}{4}+\varepsilon\right\}
\quad    \text{ and } \quad
\nu^*\big(N_{\mathcal{G}^*}(B)\big) \ge \min\left\{\left(1 + \varepsilon\right)\nu^*(B),\frac{1}{4}+\varepsilon\right\}.
\]
In particular, $\varepsilon=\min\left\{\frac{3}{20},\frac{\Phi_{\textrm{st}}(\mathcal{G})}{4\overline{\deg}}\right\}$ satisfies this.
\end{lemma}
\begin{proof}
For ease of notation we will write $B'=N_{\mathcal{G}^*}(B)$, $A'=N_{\mathcal{G}^*}(A)$, and $E=E(\mathcal{G})$. For $B\subseteq X_2$, the set $B'$ consists exactly of those edges of $\mathcal{G}$ that have at least one vertex in $\pi(B)$. Consequently,

\begin{align*}
\nu^*(B')=\frac{1}{2\nu_E(E)}\nu_E(B')&=\frac{\frac{1}{2}\int_{\pi(B)} \deg(u)d\nu(u)+\frac{1}{2}\int_{\pi(B)'\setminus \pi(B)}\deg_{\pi(B)}(u)d\nu(u)}{2\nu_E(E)}\\
&\geq\frac{\frac{1}{2}\int_{\pi(B)}\deg(u)d\nu(u)}{\int_{V(\mathcal{G})} \deg(v)d\nu(v)}+\frac{\Phi_{\textrm{st}}\min\left\{\nu_{\textrm{st}}(\pi(B)),1-\nu_{\textrm{st}}(\pi(B))\right\}}{2\overline{\deg}}\\
&\begin{cases}
\geq \nu^*(B)+\frac{\Phi_{\textrm{st}}}{\overline{\deg}}\cdot\frac{\nu_{\textrm{st}}(\pi(B))}{2} \geq \left(1+\frac{\Phi_{\textrm{st}}}{\overline{\deg}}\right)\nu^*(B) & \textrm{ if } \nu_{\textrm{st}}(\pi(B))\leq\frac{1}{2}\\
=\frac{1}{2}\nu_{\textrm{st}}(\pi(B))+\frac{\Phi_{\textrm{st}}\left(1-\nu_{\textrm{st}}(\pi(B))\right)}{2\overline{\deg}}\geq\frac{1}{4}+\frac{\Phi_{\textrm{st}}}{4\overline{\deg}}& \textrm{ if } \nu_{\textrm{st}}(\pi(B))\geq\frac{1}{2},
\end{cases}
\end{align*}
where $\pi(B)'=N_{\mathcal{G}}\big(\pi(B)\big)$.

Now let us consider $A\subseteq X_1$. In this case, $A'$ is all possible lifts of the vertices induced by $A$ in $\mathcal{G}$. That is, if $S\subseteq V(\mathcal{G})$ is the set of vertices that are incident to at least one edge from $A$, then $A'=\pi^{-1}(S)$. Thus
\begin{align*}
    \nu^*(A')&=\frac{1}{\int_{V(\mathcal{G})}\deg(v)d\nu(v)}\int_{S} |\pi^{-1}(u)| d\nu(u)=\frac{1}{2\nu_E(E)}\int_{S} \frac{1}{2}\deg(u) d\nu(u)  \\
    &\geq\frac{1}{2\nu_E(E)}\left(\frac{1}{2}\int_S\deg_S(u)d\nu(u)+\frac{1}{2}\Phi_{\textrm{st}}\min\left\{\nu_{\textrm{st}}(S),1-\nu_{\textrm{st}}(S)\right\}\right)\\
    &\begin{cases}
    \ge \frac{\frac{1}{2}\int_S\deg_S(u)d\nu(u)}{2\nu_E(E)}\left(1+\frac{\Phi_{\textrm{st}}}{\overline{\deg}}\right)\geq \nu^*(A)\left(1+\frac{\Phi_{\textrm{st}}}{\overline{\deg}}\right) & \textrm{ if } \nu_{\textrm{st}}(S)\leq\frac{1}{2}\\
    \ge \frac{\frac{1}{2}\int_S\deg_S(u)d\nu(u)}{2\nu_E(E)}\left(1+\frac{\Phi_{\textrm{st}}(1-\nu_{\textrm{st}}(S))}{\int_S\deg(u)d\nu(u)}\right)\geq\nu^*(A)\left(1+\frac{\Phi_{\textrm{st}}}{\overline{\deg}}\cdot\frac{1-\nu_{\textrm{st}}(S)}{\nu_{\textrm{st}}(S)}\right) & \textrm{ if } \nu_{\textrm{st}}(S)\geq\frac{1}{2}.
    \end{cases}
\end{align*}
We hence have that $\nu^*(A')\geq\left(1+\frac{\Phi_{\textrm{st}}}{4\overline{\deg}}\right)\nu^*(A)$ for all $A\subseteq X_1$ such that $\nu_{\textrm{st}}(S)\leq\frac{4}{5}$. Moreover, $\nu^*(A')=\frac{1}{2}\nu_{\textrm{st}}(S)$, which means that $\nu_{\textrm{st}}(A')\geq\frac{1}{4}+\frac{3}{20}$ whenever $A\subseteq X_1$ is such that $\nu_{\textrm{st}}(S)\geq\frac{4}{5}$.
\end{proof}

\begin{proof}[Proof of Theorem~\ref{thm:non-amenable_balanced_orientation}]
We aim to find a factor-of-iid balanced orientation of the quasi-transitive graph $G$, that is we aim to find a measurable balanced orientation (up to nullsets) in its Bernoulli graphing $(\mathcal{G},\nu)$.

The spectrum of the Markov operator $\mathcal{M}_{\mathcal{G}}$ restricted to $L^2_0\big(V(\mathcal{G}),\nu_{\textrm{st}}\big)$ is bounded away from $1$ (though not necessarily bounded away from $-1$). This is given for $\mathcal{M}_\mathcal{G}$ with non-bipartite $M_T$ by Theorem~\ref{thm:non_bipart_spectral_gap}. For $\mathcal{M}_\mathcal{G}$ with bipartite $M_T$, we deduce this by observing that $L^2_0\left(V(\mathcal{G}),\nu_{\textrm{st}}\right)$ can be written as the direct sum $L^2_{00}(V\left(\mathcal{G}),\nu_{\textrm{st}}\right) \oplus \langle \mathbbm{1}_{X_1}-\mathbbm{1}_{X_2}\rangle$ and applying Theorem~\ref{thm:bipart_spectral_gap}.

By a standard argument this spectral gap ``at the top of the spectrum'' implies that $\mathcal{G}$ has positive Cheeger constant. See e.g.\ \cite[Proposition 3.3.6]{kowalski2019introduction} for a formulation and proof for finite graphs that generalizes to graphings (with the appropriate vertex- and edge measures). 
Consequently by Lemma~\ref{lemma:cheeger_implies_aux_bipart_expansion} the auxiliary graphing $\mathcal{G}^*$ has bipartite vertex expansion, which means it has a measurable perfect matching $M$ by Theorem~\ref{thm:bipartite_graphing_perfect_matching}. Then $M$ defines a measurable balanced orientation of $\mathcal{G}$ as described after Definition~\ref{def:aux_graphing}.
\end{proof}

\section{Other decorations}\label{section:decorations}

\subsection{Schreier decorations of $T_{2d}$}

In this section we prove the four items of Proposition~\ref{prop:observations}.

We start by pointing out that for $T_{2d}$, there are in fact unique $\Aut(T_{2d})$-invariant measures $\mu_{\textrm{bo}}$ and $\mu_{\textrm{Sch}}$ on the spaces ${\tt BalOr}(T_{2d})$ and ${\tt Sch}(T_{2d})$ respectively. The reason is that both the balanced orientation and Schreier decoration are essentially unique on $T_{2d}$, meaning that $\Stab_{\Aut(T_{2d})}(o)$ acts transitively on both ${\tt BalOr}(T_{2d})$ and ${\tt Sch}(T_{2d})$. Here $o$ denotes an arbitrary root vertex in $T_{2d}$.

One can construct $\mu_{\textrm{bo}}$ and $\mu_{\textrm{Sch}}$ by starting at $o$ and defining the balanced orientation or Schreier decoration on the incident edges uniformly at random. Then continue moving radially outwards through the vertices of $T_{2d}$, always extending the structure to the $2d-1$ outwards edges where it is not yet defined, doing so by choosing uniformly randomly among the possible extensions, independently at each vertex. 

Note that $\mu_{\textrm{bo}}$ is a factor of $\mu_{\textrm{Sch}}$, simply by forgetting the colors. In fact, there is an intermediate object, which we can obtain from $\mu_{\textrm{Sch}}$ by forgetting the order of the last two colors $c_{d-1}$ and $c_d$. This gives $\mu_{\textrm{Sch}^*}$, the unique invariant measure on ${\tt Sch}^*(T_d)$, the space of Schreier decorations of $T_{2d}$ with the colors $\{c_{d-1}, c_d\}$ unordered. So the more detailed picture is that $\mu_{\textrm{bo}}$ is a factor of $\mu_{\textrm{Sch}^*}$, which is itself a factor of $\mu_{\textrm{Sch}}$. 

Theorem~\ref{thm:non-amenable_balanced_orientation} implies that $\mu_{bo}$ is a factor of iid. For $d>1$, one could show that $\mu_{\textrm{Sch}}$ is a factor of iid if $T_d$ had a factor of iid proper edge $d$-coloring. (However, the existence of such a coloring is an open question \cite{lyons2017factors}.) 

\begin{proof}[Proof of \ref{itm:if_edge_coloring_then_sch_dec}]
A balanced orientation of $T_{2d}$ gives rise to a decomposition of the edges into infinitely many edge-disjoint $d$-regular subtrees, with each subtree having either only incoming or only outgoing edges at every vertex it covers. Each vertex is covered by exactly two such $d$-regular subtrees. 

We construct a balanced orientation (and the resulting decomposition) as a factor of iid by Theorem~\ref{thm:non-amenable_balanced_orientation}. Furthermore, we can assume that each vertex $v$ still has two independent uniform random labels $l_{\mathrm{in}}(v)$ and $l_{\mathrm{out}}(v)$ to be used in each of the two subtrees covering it. Then by using the assumed factor-of-iid proper edge $d$-coloring on each subtree, we obtain a Schreier decoration.
\end{proof}

In \cite{lyons2017factors}, Lyons presents a partial result towards constructing the unique invariant measure $\mu_{\textrm{col}}$ on proper edge colorings of $T_d$ with $d$ colors as a factor-of-iid. He obtains a factor-of-iid proper edge coloring, but with the last two colors being unordered. This allows us to prove part~\ref{itm:colourblind_Schreier}, which states that even $\mu_{\textrm{Sch}^*}$ is a factor of iid.

\begin{proof}[Proof of \ref{itm:colourblind_Schreier}]
We follow the construction of part~\ref{itm:if_edge_coloring_then_sch_dec}, and use the factor-of-iid proper edge coloring with two colors unordered from \cite{lyons2017factors} on the $d$-regular subtrees. To complete the construction, at every vertex of the tree, we have to match the colors of the $\{c_{d-1}, c_d\}$-colored incoming edges to the two outgoing $\{c_{d-1}, c_d\}$-colored edges. So each vertex chooses a random bijection between these incoming and outgoing edges, placing the paired edges in the same color class from $\{c_{d-1}, c_d\}$. 
\end{proof}

Notice that the map forgetting the order of colors from ${\tt Sch}(T_{2d})$ to ${\tt Sch}^*(T_{2d})$ is a $2$-to-$1$ cover. In a sense, we are only lacking a coin flip to find a Schreier decoration. However, this is exactly the kind of randomness that cannot be used when constructing factors of iid -- vertices far away cannot generate a common random value because of correlation decay. In part~\ref{itm:orientation_not_finishable}, we explicitly show that ``finishing the construction'' starting from a balanced orientation of $T_4$ is not possible.

\begin{proof}[Proof of \ref{itm:orientation_not_finishable}]
Take two oriented edges $\vec{e}$ and $\vec{f}$ of $\vec{T_{4}}$ that are the first and the last edge on a path on which the orientation is alternating. Coloring one of them determines the color of the other in a Schreier decoration that respects the orientation. If the path consists of an odd number of edges then 
$\vec{e}$ and $\vec{f}$ have to have the same color. 

On the other hand, the action of $\Aut(\vec{T_4})$ is edge-transitive, which implies that if we pick $\vec{e}$ and $\vec{f}$ further and further apart, the correlation between their colors must decay. Hence, there can be no factor-of-iid Schreier decoration respecting the orientation.  
\end{proof}

Finally we prove part~\ref{itm:if_sch_dec_then_bigger_sch_dec}, namely that if $\mu_{\textrm{Sch}}$ is a factor of iid for $T_{2d}$, then it is a factor of iid also for $T_{2d+2i}$ for all $i\in\mathbb{N}$.

\begin{proof}[Proof of \ref{itm:if_sch_dec_then_bigger_sch_dec}]

Let us first construct two disjoint factor-of-iid perfect matchings on $T_{2d+2}$ as in \cite{LyonsNazarov}. 
Then after disregarding the edges in these matchings, we are left with infinitely many copies of $T_{2d}$, in which we can find, by assumption, a factor-of-iid Schreier decoration with colors $c_1,\dots,c_d$. Let us now in the tree $T_{2d+2}$ disregard the edges colored with $c_1$, so that we again are left with infinitely many $T_{2d}$-s. We delete the $\{c_2,\dots,c_d\}$, orientation, and matching decorations in these trees, and construct on them, anew, a Schreier decoration with colors $\{c_2,\dots,c_{d+1}\}$. Together with the edges decorated with $c_1$, this gives a Schreier decoration of the tree $T_{2d+2}$.
\end{proof}

\subsection{A connection to measured group theory} \label{subsec:proper_colouring_vs_Schreierization}
Part~\ref{itm:colourblind_Schreier} of Proposition~\ref{prop:observations} also has the following interpretation.

The $2d$-regular tree is the Cayley graph of $F_d$, the free group on $d$ generators, but also of the group $(\mathbb{Z}/2\mathbb{Z})^{*2d}$, the $2d$-fold free product of $(\mathbb{Z}/2\mathbb{Z})$ with itself. A Schreier decoration corresponds to an action of $F_d$, while a proper edge coloring corresponds to an action $(\mathbb{Z}/2\mathbb{Z})^{*2d}$. 

Let $\Gamma = (\mathbb{Z}/2\mathbb{Z})^{*2d}$. Consider the \emph{Bernoulli shift} $F_{d} \acts \big( [0,1]^{F_{d}}, {\tt u}^{F_{d}} \big)$, and similarly $\Gamma \acts \big([0,1]^{\Gamma}, {\tt u}^{\Gamma} \big)$. Let $S$ and $T$ denote the standard generating sets of $F_2$ and $\Gamma$ respectively. 

One can ask whether the two Bernoulli shifts are equivalent in the strong sense that there exists a measure-preserving bijection $\Phi: [0,1]^{F_{d}} \to [0,1]^{\Gamma}$ such that (on a subset of measure 1) whenever $s. \omega = \omega'$ for $\omega, \omega' \in [0,1]^{F_{d}}$, and $s \in S$, then there is some $t \in T$ such that $t. \Phi(\omega) = \Phi(\omega')$. (Note that this is much stronger than Orbit Equivalence, we require $t$ to be from the finite generating set $T$. We require that the distances defined by the word length on the orbits are preserved.)

As far as the authors are aware, this question is open. The existence of such $\Phi$ would imply that $F_d$ has a p.m.p.\ action on $[0,1]^{\Gamma}$ that defines the same distance on orbits as $\Gamma$ and vice versa. So disproving the equivalence 
could be achieved by showing that one of these actions does not exist. This is a fruitful approach when considering the same problem for groups with Cayley graphs isomorphic to the square lattice. 

The results of \cite{lyons2017factors} and part~\ref{itm:colourblind_Schreier} of  Proposition~\ref{prop:observations} respectively say that $\Gamma$ acts on a $2$-cover of $[0,1]^{F_d}$ defining the same distance on orbits, and $F_d$ acts on a $2$-cover of $[0,1]^{\Gamma}$ and defines the same distance on orbits.

\subsection{Decorations of $G^*$} \label{subsec:decorations_of_G*}
In this subsection, we further study the connection between balanced orientations of $G$ and perfect matchings of the auxiliary graph $G^*$.

We will first finish proving the equivalence of a balanced orientation of $G$ with a perfect matching on $G^*$ started in Section~\ref{sec:balanced_orientation}, and then use the perfect matching to construct Schreier decorations and proper edge colorings of $G^*$.

\begin{lemma}\label{lemma:balancedori_is_perfmatch}
Let $G$ be a simple graph with all degrees even. There is a (finitary) $\Aut(G)$-factor of iid balanced orientation of $G$ if and only if there is a (finitary) $\Aut(G^*)$-factor of iid perfect matching.
\end{lemma}

\begin{proof}
Suppose $G^*$ has a factor-of-iid perfect matching. Given random labels on $V(G)$, we can deterministically produce labels on $V(G^*)$ as we will describe below. We use the factor-of-iid perfect matching to deterministically compute matching $M$ in $G^*$, which again deterministically defines a balanced orientation of $G$. As all the steps are $\Aut(G)$-equivariant, their composition is a factor-of-iid balanced orientation of $G$.

By decomposing our original labels, we can assume that we have $\frac{3}{2}\deg(v)$ independent random labels at each $v \in V(G)$ at the beginning. We make each $v$ give one of these labels to all the $v_i$ as well as all $x_e$ for edges $e$ incident to $v$. Then each $x_e$ takes the two labels it got from its endpoints and composes them to get a label. This way each vertex of $V(G^*)$ obtains a label. The joint distribution of these labels is uniform iid, which completes the construction in this direction.

On the other hand, suppose $G$ has a factor-of-iid balanced orientation. Without loss of generality, we will assume that $G$ is connected. If $G\neq P$ and $G \neq C_k$, then any $y \in V(G^*)$ can determine whether it is of edge-type or vertex-type. Also, if $y$ is of vertex-type (say $y=v_i$), it can identify all other vertices of $G^*$ that correspond to the same vertex of $G$ as $y$ (all vertices of the form $v_j$, $j \in [\deg(v)/2]$). If the $v_j$, $j \in [\deg(v)/2]$ compose their labels to get a label $l(v)$ for each $v \in V(G)$, then any $y \in V(G^*)$ of vertex-type can simulate the factor-of-iid balanced orientation on ``its neighborhood in $G$''. The balanced orientation determines which $\deg(v)/2$ vertices of the form $x_{vu}$ get matched to the $v_i$. The $v_i$ can together choose the matching between $\{v_i ~|~ i \in [\deg(v)/2]\}$ and $\{x_{vu} ~|~ uv \textrm{ is oriented towards } v \textrm{ in } G\}$ randomly, yielding a factor-of-iid perfect matching of $G^*$. 

If $G=P$ is the bi-infinite path then $G^*=G$ and it has neither factor of iid perfect matching nor balanced orientation. If $G=C_k$ for some $k\in\mathbb{N}$ then $G^*=C_{2k}$, and so there is both a balanced orientation on $G$ and a perfect matching on $G^*$.
\end{proof}

\begin{remark}
From a more algebraic point of view, in the proof above, we use the fact that $\Aut(G)$ acts on $[0,1]^{V(G^*)}$. There is a natural embedding $\varphi: \Aut(G) \to \Aut(G^*)$, which in turn defines the translation action of $\Aut(G)$ on $[0,1]^{V(G^*)}$. By decomposing and combining the labels as explained, we have in fact shown that $\Aut(G) \acts \big([0,1]^{V(G^*)}, {\tt u}^{V(G^*)}\big)$ is a factor of $\Aut(G) \acts \big( [0,1]^{V(G)}, {\tt u}^{V(G)}\big)$. We can then utilize the existence of an $\Aut(G^*)$-factor of iid perfect matching of $G^*$ and the correspondence with balanced orientations of $G$ to finish the proof by composing the appropriate factor maps.

In the other direction, we aim to build a factor map from $\Aut(G^*) \acts \big([0,1]^{V(G^*)}, {\tt u}^{V(G^*)}\big)$ to $\Aut(G^*) \acts \big( {\tt PM}(G^*), \mu_{\textrm{pm}} \big)$ through the factor from $\Aut(G) \acts \big( [0,1]^{V(G)}, {\tt u}^{V(G)}\big)$ to $\Aut(G) \acts \big( {\tt BalOr}(G), \mu_{\textrm{bo}} \big)$. But in order to do that we have to consider $ \big([0,1]^{V(G)}, {\tt u}^{V(G)}\big)$ as an $\Aut(G^*)$-space. This is possible exactly when $G$ has a vertex of degree at least $4$, or equivalently when vertices of $G^*$ can determine their type. In this case every element of $\Aut(G^*)$ is an element of $\Aut(G)$ up to permuting the sets $\{v_i ~|~ i \in [\deg(v)/2]\}$. 
\end{remark}

As an immediate corollary, we get factor-of-iid perfect matchings on $G^*$.

\begin{corollary}
Let $G$ be a unimodular, quasi-transitive, non-amenable graph with all degrees even. Then $G^*$ is a unimodular, quasi-transitive, non-amenable graph that has a factor-of-iid perfect matching.
\end{corollary}
\begin{proof}
Follows from Theorem~\ref{thm:non-amenable_balanced_orientation} and Lemma~\ref{lemma:balancedori_is_perfmatch}.
\end{proof}

Note that even though we obtained the factor-of-iid balanced orientation in Theorem~\ref{thm:non-amenable_balanced_orientation} through perfect matchings, there we used the auxiliary graphing $\mathcal{G}^*$ of the Bernoulli graphing $\mathcal{G}$ (of $G$). Whereas here we claim that the Bernoulli graphing of the graph $G^*$ has a measurable balanced orientation.

We are now also ready to prove Proposition~\ref{prop:onG*}. For the reader's convenience we restate it here.

\begin{T1}
For every $2d$-regular graph $G$, the bipartite graph $G^*$ is also $2d$-regular, and the following are equivalent.
\begin{enumerate}
\item $G^*$ has got a factor-of-iid proper edge $2d$-coloring.
\item $G^*$ has got a factor-of-iid perfect matching.
\item $G^*$ has got a factor-of-iid Schreier decoration.
\end{enumerate}
Moreover, if any of these is a finitary factor, the others are too.
\end{T1}

Even though proving three implications would be enough, we show five to emphasize the techniques that could be used more widely for other suitable bipartite graphs too.

\begin{proof}
Let $v$ be a vertex of $G$ whose neighbors are $u^1,\dots,u^{2d}$. Then for every $i\in\left[\frac{\deg(v)}{2}\right]=[d]$, the neighbors of $v_i$ in $G^*$ are exactly $x_{vu^1},\dots,x_{vu^{2d}}$. Also for any edge $uv$ in $G$, the neighbors of $x_{uv}$ in $G^*$ are $u_1,\dots,u_d,v_1,\dots,v_d$, and so $G^*$ is also $2d$-regular. Let us denote by $A_V$ the set of vertices of $G^*$ that are of vertex-type, and by $A_E$ the set of vertices of edge-type.

\vspace{5px}
\noindent\textbf{$1\implies2$}.
Choose one of the $2d$ color classes to obtain a perfect matching.

\vspace{5px}
\noindent\textbf{$3\implies2$}. 
Every finite bipartite $2d$-regular graph has a perfect matching and choosing one at random is a factor-of-iid process, so we can assume $G$ is infinite. $P^*=P$ does not admit a factor-of-iid Schreier decoration, so let us suppose that $d\geq2$. Then as in the proof of Lemma~\ref{lemma:balancedori_is_perfmatch}, every vertex can determine whether it belongs to $A_V\subset V(G^*)$ or $A_E\subset V(G^*)$. To obtain a perfect matching, let us fix a color $c$ of the decoration and let each $x\in A_V$ pick the outgoing edge of color $c$ and each $x\in A_E$ the incoming edge of color $c$. 

\vspace{5px}
\noindent\textbf{$3\implies1$}.
Suppose the Schreier decoration uses colors $c_1,\dots,c_d$ and that we want to produce proper coloring with colors $c'_1,\dots,c'_{2d}$. Similarly as in the proof of $3\implies2$, let each edge of color $c_i$ going from $A_E$ to $A_V$ get color $c'_{2i}$ and each edge of color $c_i$ going from $A_V$ to $A_E$ the color $c'_{2i-1}$.

\vspace{5px}
\noindent\textbf{$2\implies3$}.
For every $v\in V(G)$, the $d$ copies of $v$ in $G^*$ together with the $d$ vertices they are matched to induce a $K_{d,d}$. Let us note that the collection of these $K_{d,d}$-s is vertex-disjoint. Let us randomly pick a proper edge $d$-coloring on each of these $K_{d,d}$-s and orient all their edges from $A_E$ to $A_V$. After removing the decorated edges, we are again left with a collection of vertex-disjoint $K_{d,d}$-s, this time in each of which one part is formed by $v_1,\dots,v_d$ for some $v\in V(G)$ and the other by the neighbors of $v_i, i\in[d]$ that are matched towards some $u_j$ where $uv\in E(G)$. Each of these $K_{d,d}$-s again picks a proper $d$-coloring at random, but this time we will orient each edge from $A_V$ to $A_E$.

\vspace{5px}
\noindent\textbf{$1\implies3$}.
Suppose $E(G^*)$ is colored with $c_1,\dots,c_{2d}$. Let every edge of color $c_1,c_3,\dots$ retain it and become oriented from $A_E$ to $A_V$. Then all edges of a color $c_i,i\in[d]$ will get recolored to $c_i$ and get oriented from $A_V$ to $A_E$.
\end{proof}

\section{Open questions}

\label{section:open_questions}

\begin{question}\label{qtn:fiid_sch}
Is the unique $\Aut(T_{2d})$-invariant measure $\mu_{\textrm{Sch}}$ on ${\tt Sch}(T_{2d})$ a factor of iid?
\end{question}

We believe that this question, which has already been asked in \cite{Ball} for the case $d=2$, is the most natural and important one at this point. A very similar question asking for any Cayley diagram, not just of $F_d$, as a factor of iid on the regular tree was asked by Thornton \cite[Problem 4.16]{thornton2020factor}.
Our positive examples of Schreier decorations in \cite{ourselves} so far seem fundamentally different from the tree in the sense that none of them even have infinite monochromatic paths, which would be automatic on $T_{2d}$. This leads us to the following question (also included in \cite{ourselves}).
\begin{question} \label{qtn:sch_infinite_monochrom}
Is there a factor-of-iid Schreier decoration on a transitive graph that has infinite monochromatic paths with positive probability?
\end{question}

The Schreier decoration of $T_{2d}^*$ obtained from a factor-of-iid perfect matching according to the \textbf{$2\implies3$} part of Proposition~\ref{prop:onG*} has infinite monochromatic paths, but $T_{2d}^*$ is not transitive.

The following is the question discussed in subsection~\ref{subsec:proper_colouring_vs_Schreierization}. We encountered it during personal communication with Matthieu Joseph. 

\begin{question} \label{qtn:Bernoulli_isometry}
Is there a measurable bijection $\Phi$ between the Bernoulli shifts of the free group $F_d$ and the free product $(\mathbb{Z}/2\mathbb{Z})^{*2d}$ that preserves the distance define by word length on almost all orbits?
\end{question}

Regarding our spectral result on quasi-transitive unimodular graphs, a natural question is to ask for an extension to unimodular random graphs. 
\begin{question}\label{qtn:spectral_gap}
Let $(G,o)$ be an invariantly non-amenable (a.k.a. non-hyperfinite) unimodular random rooted graph, and let $\mathcal{G}$ denote the Bernoulli graphing on $(G,o)$. Does the Markov operator $\mathcal{M}$ on $\mathcal{G}$ have spectral gap? Maybe under some stronger assumption of non-amenability? 

{\bf Addendum.} After the first version of this paper was made available online, Abért, Fraczyk, and Hayes answered Question~\ref{qtn:spectral_gap} negatively. They construct a unimodular random rooted graph that is non-amenable almost surely, but its Bernoulli graphing does not have spectral gap.
\end{question}

\bibliography{Paper_2_non-amenable}
\bibliographystyle{plain}

\end{document}